\documentclass[11pt]{article}

\usepackage{amsthm,amssymb,amsfonts,amsmath,enumerate,graphicx,color}
\usepackage{indentfirst}
\usepackage{framed,enumitem}
\usepackage[T1]{fontenc}
\usepackage[utf8]{inputenc}
\usepackage{authblk}
\usepackage{graphicx}
\usepackage{subcaption}
\usepackage{tikz}
\usepackage{tkz-graph}

\newcommand\commentout[1]{}

\newtheorem{theorem}{Theorem}[section]
\newtheorem{corollary}[theorem]{Corollary}
\newtheorem{proposition}[theorem]{Proposition}
\newtheorem{lemma}[theorem]{Lemma}

\theoremstyle{remark}

\theoremstyle{definition}
\newtheorem{definition}[theorem]{Definition}

\title{Strengthening Relationships between Neural Ideals and Receptive Fields}
\author{Angelique Morvant}%\thanks{aemorvant@gmail.com, Texas A\&M University, 400 Bizzell St, College Station, TX 77840}}
\affil[]{}
\date{March 8, 2018}

\begin{document}
\maketitle

% abstract information
\begin{abstract}
Neural codes are collections of binary vectors that represent the firing patterns of neurons. The information given by a neural code $C$ can be represented by its neural ideal $J_C$. In turn, the polynomials in $J_C$ can be used to determine the relationships among the receptive fields of the neurons. In a paper by Curto et al., three such relationships, known as the Type 1-3 relations, were linked to the neural ideal by three if-and-only-if statements. Later, Garcia et al. discovered the Type 4-6 relations. These new relations differed from the first three in that they were related to $J_C$ by one-way implications. In this paper, we first show that the converses of these new implications are false at the level of both the neural ideal $J_C$ and the larger ideal $I(C)$ of a code. We then present modified statements of these relations that, like the first three, can be related by if-and-only-if statements to both $J_C$ and $I(C)$. Using the modified relations, we uncover a new relationship involving $J_C$, $I(C)$, and the Type 1-6 relations.
\end{abstract}

%%%%%%%%%%%%%%%%%%%%%%%%%%%%%%%%%%%%%%%%%%%%%%%%%%%%%%%
%%%%%%%%%%%%%%%%%%%%%%%%%%%%%%%%%%%%%%%%%%%%%%%%%%%%%%%

\section{Introduction} 

	One of the goals of neuroscience is to determine how the firing of neurons helps the brain understand its environment. 
Some neurons are observed to fire rapidly in response to particular stimuli; the set of such stimuli is known as the neuron's 
{\em receptive field}. One aim of the study of such neurons is to determine how their firing patterns encode the relationships
among their receptive fields. For example, place cells, discovered  in 1971 by  O'Keefe and Dostrovsky, fire more rapidly when an animal is 
in certain regions, allowing the animal to navigate its environment \cite{Place_cell_discovery}. In this case, the receptive field of a place cell 
is the spatial region in which it fires, and the firing patterns of the place cells allow the brain to construct a ``map'' of these regions. We want to
understand how the brain does this.

To this end, Curto, Itskov, Veliz-Cuba, and Youngs introduced the neural ideal \cite{Neural_Rings}. Firing patterns of neurons can be recorded as collections of binary vectors known as neural codes, and the neural ideal $J_C$ of a code $C$ is a polynomial ideal that contains the same information as the code itself.  Curto et al. showed that the presence of certain types of polynomials in the neural ideal gives information about the relationships among receptive fields \cite{Neural_Rings}. For example, consider two neurons that never fire at the same time. The corresponding neural code might be $C =  \{ [1,0], [0,1], [0,0] \}$, and this code is associated with the ideal $J_C = \langle x_1x_2 \rangle$. As we will see, the fact that $x_1x_2 \in J_C$ implies that the receptive fields of neuron 1 and neuron 2 do not overlap. So if $U_1$ is the receptive field of neuron 1 and $U_2$ is the receptive field of neuron 2, $U_1 \cap U_2 = \emptyset$ is a receptive field relation that can be read off from $J_C$; it is known as a Type 1 relation. Also, it can be shown that the presence of a Type 1 relation implies that $x_1x_2 \in J_C$. Thus, the relationship between $J_C$ and the Type 1 relation is if-and-only-if.

Curto et al. found three receptive field relationships, known as the Type 1-3 relationships, that can be read off from the neural ideal. The three if-and-only-if statements relating the Type 1-3 relationships to the polynomials in the neural ideal are also loosely called the Type 1-3 relations. Later, Garcia et al. discovered three more such relations, known as the Type 4-6 relations, but proved only one direction of the relations  \cite{Grobner_bases}. It is therefore natural to ask whether the converses of any of the Type 4-6 relations might also hold.

 In the sections that follow, we will prove that the answer is no for all three relations, but that we can modify the Type 4-6 relations to be if-and-only-if statements.   Like the original Type 1-3 relations, the modified versions of the Type 4-6 relations are if-and-only-if at the level of both $J_C$ and a larger ideal $I(C)$, called the ideal of $C$.  This suggests that, at least for the purposes of receptive field relations, the ideal and neural ideal of a code are interchangeable. 

The rest of this paper will be organized as follows. In Section 2, we will introduce the background and definitions needed in the rest of the work. In Section 3, we will show by counterexample that the converses of the Type 4-6 relations do not hold. In Section 4, we will present modified versions of these relations and prove that these modified relations are if-and-only-if. Finally, in Section 5, we will discuss the implications of our results and suggest topics for future research.

%%%%%%%%%%%%%%%%%%%%%%%%%%%%%%%%%%%%%%%%%%%%%%%%%%%%%%%%%%
%%%%%%%%%%%%%%%%%%%%%%%%%%%%%%%%%%%%%%%%%%%%%%%%%%%%%%%%%%

\section{Background}
In this section we introduce neural codes, pseudo-monomials, and neural ideals, as well as the prior results on which our work is based. Our notation matches that in \cite{Neural_Rings, Grobner_bases}. We begin with neural codes.
\begin{definition}
	A {\bf neural code} (or binary code) $C$ is a set of vectors in $\mathbb{F}_2^n$. A vector $c \in C$ is called a {\bf codeword}.
\end{definition}

	Each codeword $c$ in a neural code represents a firing pattern of $n$ neurons: the $i$th component of $c$ is 1 if neuron $i$ is firing and $0$ if it is not. For example, the neural code $C = \{ [0,1,1], [1,0,1], [0,0,1], [0,0,0]\}$ consists of the codewords $c_1 = [0,1,1]$, $c_2 = [1,0,1]$, $c_3 = [0, 0, 1]$, and $c_4 = [0,0,0]$. The codeword $c_1$ tells us that neurons 2 and 3 fire together while neuron 1 does not, $c_2$ tells us that neurons 1 and 3 fire together while neuron 2 does not, $c_3$ tells us that neuron 3 fires alone, and $c_4$ tells us that none of the neurons fire.

	 Alternatively, each codeword $c \in C$ can also be represented by the set $$\text{supp(}c\text{)} = \{i \in [n] \   | \   c_i  = 1 \},$$ where $[n] := \{1,2,...,n\}.$ So in the example above, $\text{supp(}c_1\text{)} = \{2,3\}$, $\text{supp(}c_2\text{)} = \{1,3\}$, $\text{supp(}c_3\text{)} = \{3\}$, and $\text{supp(}c_4\text{)} = \emptyset$. By dropping the set notation, we can write $C$ in shorthand as $C = \{\emptyset, 3, 13, 23\}$. 

	As mentioned in the previous section, neurons such as place cells fire in specific regions of a stimulus space known as receptive fields. For a (nonempty) stimulus space $X$, let $\mathcal{U} = \{U_1, U_2, ..., U_n\}$, where each $U_i \subseteq X$ is the receptive field of neuron $i$. Then we can define the associated receptive field code as 
$$C(\mathcal{U}) = \{  c \in \mathbb{F}_2^n \ | \ (\cap_{i \in \text{supp(}c\text{)}} U_i) \setminus (\cup_{j \notin \text{supp(}c\text{)}} U_j) \}.$$
 In addition, for any point $p \in X$, we will let $c(p)$ denote the codeword such that $\text{supp(}c(p)\text{)} = \{i \in [n] \ | \ p \in U_i \}$. Going back to our example, the code $C = \{\emptyset, 3, 13, 23\}$ is the receptive field code for the set $\mathcal{U} = \{ U_1, U_2, U_3 \}$ shown in Figure \ref{fig:Ex1}.

\begin{figure}[t]
\begin{center}
 \begin{tikzpicture}[scale=1.1]        
	\draw (-2,-2) rectangle (4,2); 
	\draw (1,0) circle (1.7cm) node at (1,1) {$U_3$};
	\draw (.15,0) circle (.6cm) node {$U_1$};
	\draw (1.85,0) circle (.6cm) node {$U_2$};	
           \node (X) at (3.6,1.6) {$X$};
\end{tikzpicture}   
\end{center}
\caption{Receptive fields associated with the code $C(\mathcal{U}) = \{\emptyset, 3, 13, 23\}$.}
\label{fig:Ex1}
\end{figure}
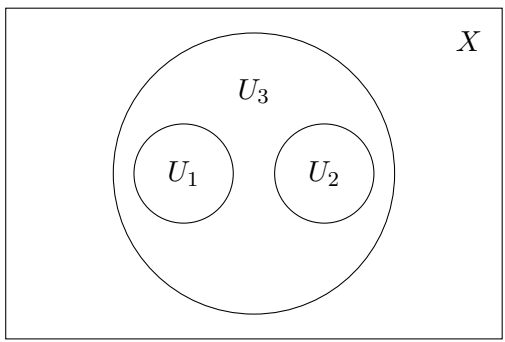

The information in a code on $n$ neurons can also be encoded in an ideal generated by pseudo-monomials in $\mathbb{F}_2[x_1,...x_n]$. 
\begin{definition}
	A {\bf pseudo-monomial} $f \in \mathbb{F}_2[x_1,...,x_n]$ is a polynomial with the form $$f = \prod_{i \in \sigma} x_i \prod_{j \in \tau} (1+x_j),$$
where $\sigma, \tau \in [n]$ and $\sigma \cap \tau = \emptyset$. 
\end{definition}
For any vector $v \in \mathbb{F}_2^n$, we define 
$$p_v = \prod_{i \in \text{supp(}v\text{)}} x_i \prod_{j \notin \text{supp(}v\text{)}}(1+x_j).$$ 
Such a pseudo-monomial is called the characteristic function for $v$ because $p_v(v) = 1$, but $p_v(x) = 0$ for all $x \in \mathbb{F}_2^n$ such that $x \neq v$. The neural ideal of a code is defined in terms of characteristic functions.

\begin{definition}
The {\bf neural ideal} $J_C$ of a code $C$ is the ideal generated by all $p_v \in \mathbb{F}_2[x_1, ...,x_n]$ such that $v \notin C$.  That is, 
$$J_C := \langle \{p_v \ | \ v \notin C \} \rangle.$$
\end{definition}
Notice that this implies that if $f \in J_C$, then $f(c) = 0$ for any $c \in C$. In \cite{Neural_Rings}, Curto et al. define the ideal of a code as follows:

\begin{definition}
Let $C \subseteq \mathbb{F}_2^n$ be a neural code. Then the {\bf ideal of $C$} is 
$$I(C) := \{f \in  \mathbb{F}_2[x_1,...,x_n] \ | \ f(c)=0 \ \forall c \in C \}.$$
\end{definition}
By this definition, $J_C \subseteq I(C)$. This fact will be important in later sections.
\par

	The polynomials in a neural ideal give information about the relationships among the receptive fields in a stimulus space.  The first three relations  (known as the Type 1-3 relations) were discovered by Curto, Itskov, Veliz-Cuba, and Youngs \cite{Neural_Rings}.  To simplify notation, we 
let 
$$x_{\sigma} = \prod_{i \in \sigma}x_i \text{ \ \ and \ \  } U_{\sigma} = \cap_{i \in \sigma} U_i$$
for any $\sigma \in [n]$. Note that if $\sigma = \emptyset$, then $x_{\sigma} = \prod_{i \in \sigma}x_i = 1$ and $U_{\sigma} = X$. \par

\begin{proposition}[Curto, Itskov, Veliz-Cuba, and Youngs] Let $X$ be a stimulus space, let $\mathcal{U} = \{U_i\}^n_{i=1}$ be a collection of sets in $X$, and consider the receptive field code $C = C(\mathcal{U})$. Then for any pair of subsets $\sigma, \tau \subseteq [n]$, we have the following receptive field relations:
\begin{itemize}[leftmargin=5em]
	\item [Type 1:] $x_{\sigma} \in J_C \Leftrightarrow U_{\sigma} = \emptyset$ (where $\sigma \neq \emptyset$). \par
	\item[Type 2:] $x_{\sigma} \prod_{i \in \tau} (1+x_i) \in J_C \Leftrightarrow U_{\sigma} \subseteq \cup_{i \in \tau} U_i$ (where $\sigma, \tau \neq \emptyset$). \par
	\item[Type 3:] $\prod_{i \in \tau}(1+x_i) \in J_C \Leftrightarrow X \subseteq \cup_{i \in \tau} U_i$ (where $\tau \neq \emptyset$), and thus $X = \cup_{i \in \tau} U_i$. \par
\end{itemize}
\end{proposition}

In addition, Garcia et al. found three more receptive field relationships \cite[Theorem 5.1]{Grobner_bases}: 

\begin{proposition}[Garcia, Garc\'{i}a Puente, et al.]
	Let $\mathcal{U} = \{U_i\}^n_{i=1}$ be a collection of sets in a stimulus space $X$. Let $C = C(\mathcal{U})$ denote the corresponding receptive field code, and let $J_C$ denote the neural ideal. Then for any subsets $\sigma_1$, $\sigma_2$, $\tau_1$, $\tau_2 \subseteq [n]$, and $m$ indices $1 \leq i_i < i_2 < ... < i_m \leq n$, with $m \geq 2$, we have receptive field relationships as follows:
\begin{itemize}[leftmargin=5em]
	\item[Type 4:] $x_{\sigma_1} \prod_{i \in \tau_1} (1+x_i) + x_{\sigma_2} \prod_{j \in \tau_2} (1+x_j) \in J_C \Rightarrow U_{\sigma_1} \cap (\cap_{i \in \tau_1} U^c_i) = U_{\sigma_2} \cap (\cap_{j \in \tau_2} U^c_j)$. \par
	\item[Type 5:] $x_{i_1} + ... + x_{i_m} \in J_C \Rightarrow U_{i_k} \subseteq \cup_{j \in [m] \setminus  \{k\}} U_{i_j}$ for all $k = 1,...,m$, and if, additionally, $m$ is odd, then $\cap^m_{k=1} U_{i_k} = \emptyset$. \par
	\item[Type 6:] $x_{i_1} + ... + x_{i_m} + 1 \in J_C \Rightarrow \cup^m_{k=1} U_{i_k} = X$.
\end{itemize}
\end{proposition}
	
	Notice that, unlike the Type 1-3 relations, the Type 4-6 relations are not stated as if-and-only-if statements. That is, knowing the relations among receptive fields in a stimulus space, we cannot use the Type 4-6 relations in their current form to conclude anything about the associated neural ideal. It is therefore natural to wonder whether each of these statements can be reversed, and in Section 3 we show that in fact none of them can. However, in Section 4 we will present modified versions of the Type 4-6 relations that are if-and-only-if statements.

%%%%%%%%%%%%%%%%%%%%%%%%%%%%%%%%%%%%%%%%%%%%%%%%%%%%%%%%%%%%%%%%
%%%%%%%%%%%%%%%%%%%%%%%%%%%%%%%%%%%%%%%%%%%%%%%%%%%%%%%%%%%%%%%%

\section{Disproving the Converses of the Type 4-6 Relations}

	In this section, we show by counterexample that none of the converses of the Type 4-6 relations hold (Theorem 3.3). The converses of the relations are stated below:

\vspace{5mm}
{\em
Let $\mathcal{U} = \{ U_i \}_{i=1}^n$ be a collection of sets in a stimulus space $X$. Let $C = C(\mathcal{U})$ denote the corresponding receptive field code, and let $J_C$ be the neural ideal of $C$. Then for any subsets $\sigma_1$, $\sigma_2$, $\tau_1$, $\tau_2 \subseteq [n]$ and $m$ indices $1 \leq i_1 < i_2 < ... < i_m \leq n$, with $m \geq 2$, we have the following:  \par
\begin{itemize}[leftmargin=8em]
	\item[Converse of Type 4:] $U_{\sigma_1} \cap (\cap_{i \in \tau_1} U^c_i) = U_{\sigma_2} \cap (\cap_{j \in \tau_2} U^c_j) \Rightarrow x_{\sigma_1} \prod_{i \in \tau_1} (1+x_i) + x_{\sigma_2} \prod_{j \in \tau_2} (1+x_j) \in J_C$ for any subsets $\sigma_1$, $\sigma_2$, $\tau_1$, $\tau_2 \subseteq [n]$. \par
	\item[Converse of Type 5:] 
		For all $m \geq 2$ indices $1 \leq i_i < i_2 < ... < i_m \leq n$,
		\begin{enumerate}
			\item $m$ is even and $U_{i_k} \subseteq \cup_{j \in [m] \setminus  \{k\}} U_{i_j}$ for all $k = 1,...,m \Rightarrow x_{i_1} + ... + x_{i_m} \in J_C$, and 
			\item $m$ is odd, $U_{i_k} \subseteq \cup_{j \in [m] \setminus  \{k\}} U_{i_j}$  for all $k = 1,...,m$, and $\cap_{k=1}^m U_{i_k} = \emptyset \Rightarrow x_{i_1} + ... + x_{i_m} \in J_C$.
		\end{enumerate}
	
	\item[Converse of Type 6:] $ \cup^m_{k=1} U_{i_k} = X \Rightarrow x_{i_1} + ... + x_{i_m} + 1 \in J_C$  for all $m \geq 2$ indices $1 \leq i_i < i_2 < ... < i_m \leq n$.
\end{itemize}
}

Again, each of the converses is false. The receptive fields for all of the counterexample codes are shown in Figure \ref{fig:Counters}. In the case of the Type 5 and 6 relations, we will first use the counterexample codes to prove that the converses of these statements do not hold {\em even when $J_C$ is replaced by the larger ideal $I(C)$}. 

\begin{figure}[t!]
\centering
    \begin{subfigure}[b]{0.32\textwidth}
	\centering
   	 \begin{tikzpicture}[scale=0.80]
		\def\firstcircle{(0,0) circle (1.5cm)}
		\def\secondcircle{(0:2cm) circle (1.5cm)}      
            	\draw \firstcircle node[left] {$U_{1}$};   %Positioning of text??
            	\draw \secondcircle node[right] {$U_{2}$};
            	\draw (-2,-2) rectangle (4,2);
            	\node (X) at (3.6,1.6) {$X$};
	\end{tikzpicture}
        \caption{Counterexample for Type 4}
        \label{fig:T4C}
    \end{subfigure}
    \hfill
    \begin{subfigure}[b]{0.32\textwidth}
         \centering
         \begin{tikzpicture}[scale=0.80]
	       \def\firstcircle{(-0.5,0.4) circle (1.2cm)}
	       \def\secondcircle{(1.0,0.4) circle (1.2cm)}     
	       \def \thirdcircle{(3,-1.0) circle (0.7cm)} 
	       \def \keyrec{(-1.2,-2.6) rectangle (-.9,-2.3)}

	\begin{scope}[fill opacity=1]
		\fill[lightgray] \firstcircle;
		\fill[lightgray] \secondcircle;
    		\fill[lightgray] \thirdcircle;
		\fill[lightgray] \keyrec;
    	\end{scope}

                    \draw \firstcircle node[left] {$U_{2}$};   %Positioning of text??
                    \draw \secondcircle node[right] {$U_{3}$};
	         \draw \thirdcircle node {$U_4$};
                    \draw (-2,-2) rectangle (4,2);
	         \draw \keyrec;
                    \node (X) at (3.6,1.6) {$X$};
	         \node (Y) at (-.4,-2.5) {$U_1$};
	\end{tikzpicture}
        \caption{Counterexample for Type 5}
        \label{fig:T5C}
  \end{subfigure}
  \hfill
   \begin{subfigure}[b]{0.32\textwidth}
        \centering
        \begin{tikzpicture}[scale=0.80]         
	 \fill[lightgray] (0,-2) rectangle (2,2); %middle
	 \draw (-2,-2) rectangle (4,2); 
	 \draw (-2,-2) rectangle (0,2); %left
	 \draw (2,-2) rectangle (4,2); %right
            \node (X) at (3.6,1.6) {$X$};
	 \node (L1) at (-1,0) {$U_1$};
	 \node (L2) at (1,0) {$U_1 \cap U_2$};
	 \node (L3) at (3,0) {$U_2$};
       \end{tikzpicture}
        \caption{Counterexample for Type 6}
        \label{fig:T6C}
    \end{subfigure}	
\caption{Receptive fields for which the converses of the Type 4, 5, and 6 relations do not hold.}
\label{fig:Counters}
\end{figure}
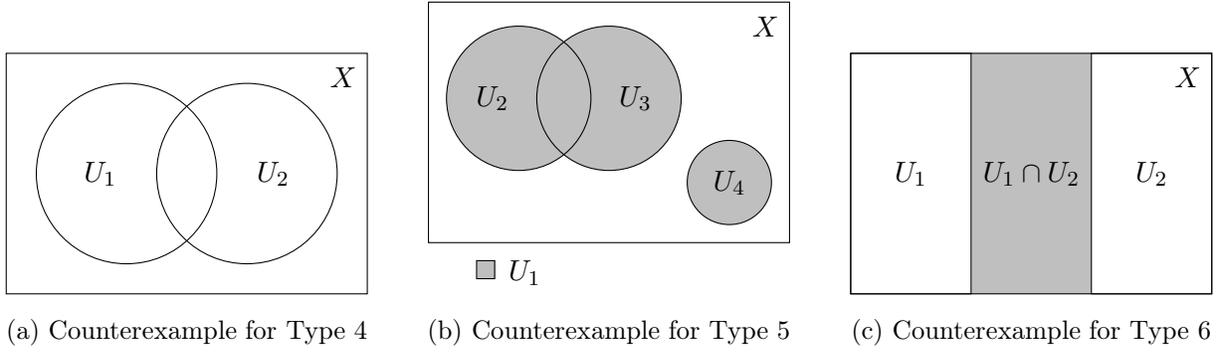

\begin{figure}[t!]
\centering
        \begin{tikzpicture}[scale=0.80]
	\def\firstcircle{(-0.5,0.4) circle (1.2cm)}
	\def\secondcircle{(1.0,0.4) circle (1.2cm)}     
	\def \thirdcircle{(2.9,-1.0) circle (0.9cm)} 
	\def \fourthcircle{(2.9,-1.3) circle (.4cm)}
	\def \keyrec{(-1.2,-2.6) rectangle (-.9,-2.3)}

	\begin{scope}[fill opacity=1]
		\fill[lightgray] \firstcircle;
		\fill[lightgray] \secondcircle;
    		\fill[lightgray] \thirdcircle;
		\fill[lightgray] \fourthcircle;
		\fill[lightgray] \keyrec;
    	\end{scope}

            \draw \firstcircle node[left] {$U_{2}$};   %Positioning of text??
            \draw \secondcircle node[right] {$U_{3}$};
	 \draw \thirdcircle node[above] {$U_4$};
	 \draw \fourthcircle node {$U_5$};
            \draw (-2,-2) rectangle (4,2);
	 \draw \keyrec;
            \node (X) at (3.6,1.6) {$X$};
	 \node (Y) at (-.4,-2.5) {$U_1$};
        \end{tikzpicture}
\caption{Receptive fields for which the converse of the Type 5 relation is false when $m$ is odd.}
\label{fig:Counter5Odd}
\end{figure}
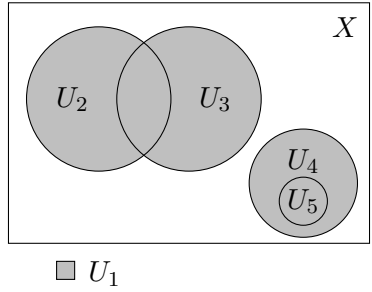

\begin{lemma}
For each of the following relations, there exists a collection of sets $\mathcal{U} = \{U_i\}^n_{i=1}$ in a stimulus space $X$ (with corresponding neural code $C = C(\mathcal{U})$) and $m \geq 2$ indices $1 \leq i_i < i_2 < ... < i_m \leq n$ such that the relation does {\bf not} hold: \par
\begin{itemize}[leftmargin=8em]
	\item[Converse of Type 5:] 
		\begin{enumerate}
			\item $m$ is even and $U_{i_k} \subseteq \cup_{j \in [m] \setminus  \{k\}} U_{i_j}$ for all $k = 1,...,m \Rightarrow x_{i_1} + ... + x_{i_m} \in I(C)$, and 
			\item $m$ is odd, $U_{i_k} \subseteq \cup_{j \in [m] \setminus  \{k\}} U_{i_j}$  for all $k = 1,...,m$, and $\cap_{k=1}^m U_{i_k} = \emptyset \Rightarrow x_{i_1} + ... + x_{i_m} \in I(C)$.
		\end{enumerate}
	\item[Converse of Type 6:]  $\cup^m_{k=1} U_{i_k} = X \Rightarrow  x_{i_1} + ... + x_{i_m} + 1 \in I(C)$
\end{itemize}
\end{lemma}
\begin{proof} 
	{\em Converse of Type 5:}  To disprove the converse of the Type 5 relation, it is enough to prove that the first implication is false. To this end, let $C = \{\emptyset, 12, 13, 14, 123 \}$. The corresponding stimulus space is shown in Figure \ref{fig:T5C}.  From the figure, we can see that $U_{i_k} \subseteq \cup_{j \in [4]\setminus \{k\}} U_{i_j}$, where $i_1=1,...,i_4=4$.  However, if $f := x_1+x_2+x_3+x_4$, then for any point $p \in (U_1 \cap U_2 \cap U_3) \setminus U_4$, $f(c(p)) = (x_1 + x_2 + x_3 + x_4)(c(p)) = 1+1+1+0 \equiv 1 \mod 2$, so $f=x_1+x_2+x_3+x_4 \notin I(C)$. \par
	There are also counterexamples for the second implication. For instance, the code $C' = \{\emptyset, 12, 13, 14, 123, 145\}$ (shown in Figure \ref{fig:Counter5Odd}) with $i_1= 1, ..., i_5=5$ satisfies $U_{i_k} \subseteq \cup_{j \in [5]\setminus \{k\}} U_{i_j}$ for all $k=1,...,5$, and $\cap_{k=1}^5 U_{i_k} = \emptyset$. But, again, if $g := x_1 + x_2 + x_3 + x_4+x_5$ and $p \in  (U_1 \cap U_2 \cap U_3) \setminus (U_4 \cup U_5)$, then $g(c(p)) = (x_1 + x_2 + x_3 + x_4+x_5)(c(p)) = 1+1+1+0+0 \equiv 1 \mod 2$. So $g \notin I(C')$.  \par
	{\em Converse of Type 6:} Now consider the code $C = \{1,2,12\}$ on 2 neurons pictured in Figure \ref{fig:T6C}. We compute that
$$J_C = \langle (1+x_1)(1+x_2) \rangle,$$
and we can see that $\cup_{i=1}^2 U_i = X$. If the converse of the Type 6 relation held, then it would be true that 
\begin{equation}
x_1+x_2+1 \in I(C).
\end{equation}
However, 
\begin{equation} 
(1+x_1)(1+x_2) = 1+x_1+x_2+x_1x_2 \in J_C \subseteq I(C).
\end{equation}
 Adding expressions (1) and (2) would then imply that $f:=x_1x_2 \in I(C)$. However, for a point $p \in U_1 \cap U_2$,  $f(c(p)) = 1\cdot 1 = 1$ and thus $f \notin I(C)$, a contradiction. So $x_1+x_2+1 \notin I(C)$.
\end{proof}

Now we are ready to show that none of converses of the Type 4-6 relations hold.

\begin{theorem}
For each of the following relations, there exists a collection of sets $\mathcal{U} = \{U_i\}^n_{i=1}$ in a stimulus space $X$ (with corresponding neural code $C = C(\mathcal{U})$) such that the relation does {\bf not} hold:  \par
\begin{itemize}[leftmargin=8em]
	\item[Converse of Type 4:] $U_{\sigma_1} \cap (\cap_{i \in \tau_1} U^c_i) = U_{\sigma_2} \cap (\cap_{j \in \tau_2} U^c_j) \Rightarrow x_{\sigma_1} \prod_{i \in \tau_1} (1+x_i) + x_{\sigma_2} \prod_{j \in \tau_2} (1+x_j) \in J_C$ for any subsets $\sigma_1$, $\sigma_2$, $\tau_1$, $\tau_2 \subseteq [n]$. \par
	\item[Converse of Type 5:] 
		For all $m \geq 2$ indices $1 \leq i_i < i_2 < ... < i_m \leq n$,
		\begin{enumerate}
			\item $m$ is even and $U_{i_k} \subseteq \cup_{j \in [m] \setminus  \{k\}} U_{i_j}$ for all $k = 1,...,m \Rightarrow x_{i_1} + ... + x_{i_m} \in J_C$, and 
			\item $m$ is odd, $U_{i_k} \subseteq \cup_{j \in [m] \setminus  \{k\}} U_{i_j}$  for all $k = 1,...,m$, and $\cap_{k=1}^m U_{i_k} = \emptyset \Rightarrow x_{i_1} + ... + x_{i_m} \in J_C$.
		\end{enumerate}
	
	\item[Converse of Type 6:] $ \cup^m_{k=1} U_{i_k} = X \Rightarrow x_{i_1} + ... + x_{i_m} + 1 \in J_C$  for all $m \geq 2$ indices $1 \leq i_i < i_2 < ... < i_m \leq n$.
\end{itemize}
\end{theorem}

\begin{proof}
We present a counterexample of each of the three statements, beginning with the converse of the Type 4 relation. The counterexamples for the Type 5 and 6 relations are the same as in the proof of Lemma 3.1. As we will see, we can use the same counterexamples for these two relations because $J_C \subseteq I(C)$. \par
	{\em Converse of Type 4:} Consider the code $C=\{\emptyset,1,2,12\}$ on 2 neurons pictured in Figure \ref{fig:T4C}. Also let $\sigma_1 = \tau_1 = 1$ and  $\sigma_2 = \tau_2=2$. Then $U_1  \cap U_1^c = \emptyset = U_2 \cap U_2^c$, but $J_C = \langle 0 \rangle$, so $x_1(1+x_1) + x_2(1+x_2) \notin J_C$. \par
%Consider the code $C = \{\emptyset, 1, 2, 12\}$ on 2 neurons pictured in Figure \ref{fig:T4C}. Also, let  $\sigma_1 = \tau_1 = 1$ and $\sigma_2 = \tau_2 = %2$. Then $U_1 \cap U_1^c = \emptyset = U_2 \cap U_2^c$, but $J_C = \emptyset$, so $x_1(1+x_1) + x_2(1+x_2) \notin J_C$. \par
	{\em Converse of Type 5:} Again let $C = \{\emptyset, 12, 13, 14, 123 \}$. The realization of this code is shown in Figure \ref{fig:T5C}.  From the proof of Lemma 3.1, we know that $x_1 + x_2 + x_3 + x_4 \notin I(C)$, so $x_1 + x_2 + x_3 + x_4 \notin J_C \subseteq I(C)$. This shows that the first implication is false. \par
	As before, however, we can also show that the second implication is false. The code $C' = \{\emptyset, 12, 13, 14, 123, 145\}$ again serves as a counterexample. By the proof of Lemma 3.1, $x_1 + x_2 + x_3 + x_4+x_5 \notin I(C')$, so $x_1 + x_2 + x_3 + x_4+x_5 \notin J_{C'} \subseteq I(C')$. \par
	{\em Converse of Type 6:} Finally, the proof of Lemma 3.1 tells us that for the code $C = \{1,2,12\}$ pictured in Figure \ref{fig:T6C}, $x_1+x_2+1 \notin I(C)$. Therefore, $x_1+x_2+1 \notin J_C \subseteq I(C)$.
\end{proof}

Notice that if we replace $J_C$ with $I(C)$ in the Type 4 relation, the counterexample we used in the proof of Theorem 3.2 no longer works since $x_1(1+x_1) + x_2(1+x_2) \in I(C)$. In fact, we will prove in the next section that the converse of the Type 4 relation is actually true when $J_C$ replaced with $I(C)$. This suggests that the Type 4-6 relations may be modified so that their converses hold. Possible modifications will be discussed in Section 4. 

 We finish this section by showing that the counterexample codes presented all have the smallest possible numbers of neurons.

\begin{theorem}
Let $\mathcal{U} = \{U_i\}_{i=1}^n$ be a collection of sets in a stimulus space $X$. Let $C = C(\mathcal{U})$ denote the corresponding receptive field code, and let $J_C$ denote the neural ideal. Then for any subsets $\sigma_1$, $\sigma_2$, $\tau_1$, $\tau_2 \subseteq [n]$, and $m$ indices $1 \leq i_1 < i_2 <...< i_m \leq n$, with $m \geq 2$, the following hold:
\begin{itemize}[leftmargin=8em]
	\item[Converse of Type 4:] When $n=1$, $U_{\sigma_1} \cap (\cap_{i \in \tau_1} U^c_i) = U_{\sigma_2} \cap (\cap_{j \in \tau_2} U^c_j) \Rightarrow x_{\sigma_1} \prod_{i \in \tau_1} (1+x_i) + x_{\sigma_2} \prod_{j \in \tau_2} (1+x_j) \in J_C$. \par
	\item[Converse of Type 5:]  When $n\leq 3$,
		\begin{enumerate}
			\item $m$ is even and $U_{i_k} \subseteq \cup_{j \in [m] \setminus  \{k\}} U_{i_j}$ for all $k = 1,...,m \Rightarrow x_{i_1} + ... + x_{i_m} \in J_C$, and 
			\item $m$ is odd, $U_{i_k} \subseteq \cup_{j \in [m] \setminus  \{k\}} U_{i_j}$  for all $k = 1,...,m$, and $\cap_{k=1}^m U_{i_k} = \emptyset \Rightarrow x_{i_1} + ... + x_{i_m} \in J_C$.
		\end{enumerate}

	\item[Converse of Type 6:]  When $n=1$, $\cup^m_{k=1} U_{i_k} = X \Rightarrow x_{i_1} + ... + x_{i_m} + 1 \in J_C$.
\end{itemize}
\end{theorem}

\begin{proof}
	{\em Converse of Type 4:} For $n=1$, the possible neural codes are $C_1 = \{\emptyset, 1\}$ and $C_2 = \{1\}$. \par
In the case of $C_1$, $J_C = \langle 0 \rangle$, and the possible values for $U_{\sigma_1} \cap (\cap_{i \in \tau_1} U^c_i)$ are $U_1 \cap U_1^c = \emptyset$ (when $\sigma_1 = \tau_1 = 1$), $U_1 \cap X = U_1$ (when $\sigma_1=1$ and $\tau_1 = \emptyset$), $X \cap U_1^c = U_1^c$ (when $\sigma_1 = \emptyset$ and $\tau_1 = 1$), and $X \cap X = X$ (when $\sigma_1=\tau_1 = \emptyset$). Since the possible values of $ U_{\sigma_2} \cap (\cap_{j \in \tau_2} U^c_j)$ are also $\emptyset$, $U_1$, $U_1^c$, and $X$, and since none of these four sets are equal, $U_{\sigma_1} \cap (\cap_{i \in \tau_1} U^c_i) = U_{\sigma_2} \cap (\cap_{j \in \tau_2} U^c_j)$ only when $\sigma_1 = \sigma_2$ and $\tau_1 = \tau_2$. This implies that $x_{\sigma_1} \prod_{i \in \tau_1} (1+x_i) + x_{\sigma_2} \prod_{j \in \tau_2} (1+x_j) = x_{\sigma_1} \prod_{i \in \tau_1} (1+x_i) + x_{\sigma_1} \prod_{j \in \tau_1} (1+x_j)   = 0 \in J_C$. \par
	In the case of $C_2$, $J_C  = \langle 1+x_1 \rangle$, $U_1 = X$, and $U_1^c = \emptyset$. This means that there are two choices of $\sigma_1$ and $\tau_1$ for which $U_{\sigma_1} \cap (\cap_{i \in \tau_1} U^c_i) = \emptyset$: $\sigma_1 = \tau_1=1$, or $\sigma_1 = \emptyset$ and $\tau_1 = 1$. Similarly, $U_{\sigma_1} \cap (\cap_{i \in \tau_1} U^c_i) = X$ when either $\sigma_1=1$ and $\tau_1=\emptyset$, or $\sigma_1 = \tau_1 = \emptyset$. The same is true for  $U_{\sigma_2} \cap (\cap_{i \in \tau_2} U^c_i)$, so $U_{\sigma_1} \cap (\cap_{i \in \tau_1} U^c_i) = U_{\sigma_2} \cap (\cap_{j \in \tau_2} U^c_j)$ implies that we have one of the following cases: 
\begin{itemize}[leftmargin=8em]
	\item[{\em Case 1:}] $\sigma_1 = \sigma_2$ and $\tau_1=\tau_2$ \par
	\item[{\em Case 2:}] $\sigma_1 \neq \sigma_2$ and $\tau_1 = \tau_2 = \{1\}$
	\item[{\em Case 3:}] $\sigma_1 \neq \sigma_2$ and $\tau_1 = \tau_2 = \emptyset$
\end{itemize}

	In case 1, we can use the same reasoning we used when dealing with $C_1$, so it is enough to show that the converse of the Type 4 relation holds in cases 2 and 3. For these cases, we can assume without loss of generality that $\sigma_1 = 1$ and $\sigma_2 = \emptyset$. \par
In case 2, $$x_{\sigma_1} \prod_{i \in \tau_1} (1+x_i) + x_{\sigma_2} \prod_{j \in \tau_2} (1+x_j)  = x_1(1+x_1) + (1)(1+x_1) = (1+x_1)^2 \in J_C$$
because $1+x_1 \in J_C$. \par
Similarly, in case 3, $$x_{\sigma_1} \prod_{i \in \tau_1} (1+x_i) + x_{\sigma_2} \prod_{j \in \tau_2} (1+x_j) = x_1(1)+ (1)(1) = 1+x_1 \in J_C. $$ So the implication holds. \par
	{\em Converse of Type 5:} We consider the cases where $n=1$, $n=2$, and $n=3$ separately. \par
		In the case where $n=1$, we cannot have $m \geq 2$ indices , so the implication is vacuously true. \par
		When $n=2$, in order for the left-hand side of the implication to be true, we must have $m=2$, $U_1 \subseteq U_2$, and $U_2 \subseteq U_1$, which implies that $U_1 = U_2$. This means that neither 1 nor 2 is in the associated code $C$, so $x_1(1+x_2), x_2(1+x_1) \in J_C$. Therefore, $x_1(1+x_2)+ x_2(1+x_1) = x_1+x_2 \in J_C$ \par
	Finally, when $n=3$, we can have $m=2$ or $m=3$. If $m=2$, then $U_{i_1}=U_{i_2}$ as in the $n=2$ case. If we let $i_3$ be the remaining element of $\{1,2,3\}$, then $i_1$, $i_2$, $ i_1i_3$, $i_2i_3 \notin C$, which implies that the following pseudo-monomials are in $J_C$:
\begin{align*}
x_{i_1}(1+x_{i_2})(1+x_{i_3}) &= x_{i_1} + x_{i_1}x_{i_2} + x_{i_1}x_{i_3} + x_{i_1}x_{i_2}x_{i_3} \\
x_{i_2}(1+x_{i_1}) (1+x_{i_3}) &=  x_{i_2} + x_{i_1}x_{i_2} + x_{i_2}x_{i_3} + x_{i_1}x_{i_2}x_{i_3} \\
 x_{i_1}x_{i_3}(1+x_{i_2}) &= x_{i_1}x_{i_3} + x_{i_1}x_{i_2}x_{i_3}\\ 
 x_{i_2}x_{i_3}(1+x_{i_3}) &=  x_{i_2}x_{i_3} + x_{i_1}x_{i_2}x_{i_3}
\end{align*}
Adding these four terms together gives us $x_{i_1}+x_{i_2} \in J_C$. 

 	If $m=3$, similar reasoning applies, but we also must assume that $\cap_{k=1}^3 U_{i_k} = \emptyset$. We know that the codewords $1$, $2$, $3 \notin C$ since $U_{i_k} \subseteq \cup_{j \in [3] \setminus  \{k\}} U_{i_j}$ for all $k = 1,2,3$, and we know that $123 \notin C$ because $\cap_{k=1}^3 U_{i_k} = \emptyset$ by assumption. Therefore, the following are in $J_C$:
\begin{align*}
& x_1x_2x_3 \\
x_1(1+x_2)(1+x_3) &= x_{1} + x_{1}x_{2} + x_{1}x_{3} + x_{1}x_{2}x_{3} \\
x_2(1+x_1)(1+x_3) &=  x_{2} + x_{1}x_{2} + x_{2}x_{3} + x_{1}x_{2}x_{3} \\
x_3(1+x_1)(1+x_2) &= x_{3} + x_{1}x_{3} + x_{2}x_{3} + x_{1}x_{2}x_{3}. 
\end{align*}
 Adding the four terms again gives $x_1+x_2+x_3 \in J_C$. \par
	{\em Converse of Type 6:} Again, when $n=1$, it is impossible to have $m \geq 2$ different indices, so the implication is vacuously true.
\end{proof}

%%%%%%%%%%%%%%%%%%%%%%%%%%%%%%%%%%%%%%%%%%%%%%%%%%%%%%%%%%%%%%%%%%%%
%%%%%%%%%%%%%%%%%%%%%%%%%%%%%%%%%%%%%%%%%%%%%%%%%%%%%%%%%%%%%%%%%%%%

\section{Modifying the Type 4-6 Relations}

In this section, we discuss modified versions of the Type 4-6 relations and prove that their converses hold. The modifications are marked in bold in the following result.

\begin{theorem}
Let $\mathcal{U} = \{ U_i \}_{i=1}^n$ be a collection of sets in a stimulus space $X$. Let $C = C(\mathcal{U})$ denote the corresponding receptive field code, and let $J_C$ be the neural ideal of $C$. Then for any subsets $\sigma_1$, $\sigma_2$, $\tau_1$, $\tau_2 \subseteq [n]$ {\bf such that $\sigma_1 \cup \sigma_2$ and $\tau_1 \cup \tau_2$ are disjoint}, and $m$ indices $1 \leq i_1 < i_2 < ... < i_m \leq n$, with $m \geq 2$, we have the following equivalences:
\begin{itemize}[leftmargin=8em]
	\item[Modified Type 4:] $x_{\sigma_1} \prod_{i \in \tau_1} (1+x_i) + x_{\sigma_2} \prod_{j \in \tau_2} (1+x_j) \in J_C \Leftrightarrow U_{\sigma_1} \cap (\cap_{i \in \tau_1} U^c_i) = U_{\sigma_2} \cap (\cap_{j \in \tau_2} U^c_j)$
	\item[Modified Type 5:] $x_{i_1}+...+x_{i_m} \in J_C \Leftrightarrow$ {\bf for any $\sigma \subseteq [m]$ such that $|\sigma|$ is odd}, $\cap_{k \in \sigma} U_{i_k} \subseteq \cup_{j \in [m] \setminus \sigma} U_{i_j}$ \par
	\item[Modified Type 6:]  $x_{i_1} + ... + x_{i_m} + 1 \in J_C  \Leftrightarrow \cup^m_{k=1} U_{i_k} = X$ {\bf whenever $U_{i_1},...,U_{i_m}$ are pairwise disjoint}.
\end{itemize}
\end{theorem}
\begin{proof}
	By Proposition 2.6, we know that the forward direction of each of the three implications is true, so it only remains to prove the reverse implications. \par
	{\em Modified Type 4:} Given a neural code $C$ that satisfies a relation of the form $U_{\sigma_1} \cap (\cap_{i \in \tau_1} U_i^c) = U_{\sigma_2} \cap (\cap_{j \in \tau_2} U_j^c) $, we will construct the desired sum by adding pseudo-monomials known to be in $J_C$. There are two cases: one where $\tau_1 \cup \tau_2$ is empty, and the remaining case where $\tau_1 \cup \tau_2$ is nonempty. We deal with each of these cases in turn. \par

 	 {\em Case 1 ($\tau_1 \cup \tau_2 = \emptyset$):} In this case, $U_{\sigma_1} = U_{\sigma_2}$ (we ignore the trivial cases where $\sigma_1$ or $\sigma_2$ is empty or where $\sigma_1 = \sigma_2$). Then any codeword $c$ for which $\sigma_1 \subseteq c$ but $\sigma_2 \not \subseteq c$, or $\sigma_2 \subseteq c$ but $\sigma_1 \not \subseteq c$, cannot be in $C$ (this would contradict the fact that $U_{\sigma_1} = U_{\sigma_2}$), and so $p_c = \prod_{i \in c} x_i \prod_{j \notin c}(1+ x_j) \in J_C$ for every such $c$. Letting  $C_1 = \{ c \subseteq [n] \mid \sigma_1 \subseteq c$ but $\sigma_2 \not \subseteq c \}$, $C_2 =\{ c \subseteq [n] \mid \sigma_2 \subseteq c$ but $\sigma_1 \not \subseteq c \}$, and $f = \sum_{c \in C_1 \cup C_2} p_c$, we claim that  $f = x_{\sigma_1} + x_{\sigma_2}$ and that $f \in J_C$ as desired. \par

 	 To see this, first note that $f \in J_C$ because each $p_c$ is in $J_C$ by the remarks above. To prove that $f = x_{\sigma_1}+x_{\sigma_2}$, we will consider $x_\sigma$ for all $\sigma \subseteq [n]$ and determine whether $x_{\sigma}$ appears as a summand in $f$. 

	First consider $x_\sigma$, where $\sigma_1 \not \subseteq \sigma$ and $\sigma_2 \not \subseteq \sigma$. Then $x_\sigma$ is not a term in the pseudo-monomial $p_c$ for any $c \in C_1 \cup C_2$ since all terms in such a pseudo-monomial must be of the form $x_{\sigma_1}x_{\sigma'}$ or  $x_{\sigma_2}x_{\sigma'}$ for some $\sigma' \subseteq [n]$. So $x_{\sigma}$ is not a term in $f$. \par

	If $\sigma_1 \subseteq \sigma$ but $\sigma_2 \not \subseteq \sigma$, then $x_\sigma$ is not a term in $p_c$ for any $c \in C_2$ or for any $c \in C_1$ such that $c \not \subseteq \sigma$. But $x_\sigma$ appears once as a term in $p_c$ for all $c \in C_1$ such that $c \subseteq \sigma$.  There are $2^m$ such $c$, where $m = |\sigma \setminus \sigma_1|$. Therefore, if $m \geq 1$, $x_\sigma$ will be added an even number of times in the expansion of $f$ and will therefore cancel out since we are working in $\mathbb{F}_2[x]$. We get  $m=0$ only when $\sigma = \sigma_1$, which tells us that $x_{\sigma_1}$ is a monomial term in $f$. Similar reasoning applies when $\sigma_2 \subseteq \sigma$ but $\sigma_1 \not \subseteq \sigma$. \par

	 Finally, if $\sigma_1 \cup \sigma_2 \subseteq \sigma$, then $x_\sigma$ is a term in $p_c$ for any $c \in C_1 \cup C_2$ such that $c \subseteq \sigma$. As before, we need to determine how many $c \in C_1 \cup C_2$ are contained in $\sigma$. We know there are $2^m$ codewords $c$ such that $\sigma_1 \subseteq c \subseteq \sigma$ and  $2^k$ codewords $c'$ such that $\sigma_2 \subseteq c' \subseteq \sigma$, where $m = |\sigma \setminus \sigma_1 | \geq 1$ and $k = |\sigma \setminus \sigma_2| \geq 1$ \footnote{Note that if $m =0$, then $\sigma_1 \cup \sigma_2 \subseteq \sigma = \sigma_1$, which would imply that  $\sigma_1 = \sigma_2$. We can ignore this case because in this case $x_{\sigma_1} + x_{\sigma_2} = 0$, and $0 \in J_C$ for all $C$.}. However, if we count all such $c$ and $c'$, we will also count twice the $2^l$ codewords $c''$ such that $\sigma_1\cup \sigma_2 \subseteq c'' \subseteq \sigma$. (Here $l = |\sigma \setminus (\sigma_1 \cup \sigma_2)| \geq 0$.) These codewords $c''$ are not in $C_1 \cup C_2$ and therefore must be excluded twice from our count of codewords in $C_1 \cup C_2$ contained in $\sigma$. So there are $2^m + 2^k - 2(2^l)$ codewords contained in $\sigma$ that are in $C_1 \cup C_2$. This implies that $x_\sigma$ will be added $2^m + 2^k - 2(2^l)$ times in $f$ and therefore will cancel out since $2^m + 2^k - 2(2^l)$ is even. \par
	Thus, the only terms that appear in $f$ are $x_{\sigma_1}$ and $x_{\sigma_2}$, so $f = x_{\sigma_1} + x_{\sigma_2} \in J_C$ as desired. \par

	{\em Case 2 ($\tau_1 \cup \tau_2 \neq \emptyset$):} In the case where $\tau_1 \cup \tau_2 \neq \emptyset$ and $U_{\sigma_1} \cap (\cap_{i \in \tau_1} U_i^c) = U_{\sigma_2} \cap (\cap_{j \in \tau_2} U_j^c) $, we can apply the same reasoning used in the previous case if we first define a new stimulus space $\{U_i'\}$ (with receptive field code $C'$) as follows: 
\begin{equation} \label{eqn:U_prime}
       U_i' = 
       \begin{cases}
            U_i & \text{if $i \not \in \tau_1 \cup \tau_2$} \\
            U_i^c & \text{if $i \in \tau_1 \cup \tau_2$.} \\
        \end{cases}
\end{equation}
 
Because $\sigma_1 \cup \sigma_2$ and $\tau_1 \cup \tau_2$ are disjoint and $U_{\sigma_1} \cap (\cap_{i \in \tau_1} U_i^c) = U_{\sigma_2} \cap (\cap_{j \in \tau_2} U_j^c) $, we know that $U_{\sigma_1} '\cap (\cap_{i \in \tau_1} U_i') = U_{\sigma_2}' \cap (\cap_{j \in \tau_2} U_j')$ as well. Therefore, we can say from the first case that $x_{\sigma_1 \cup \tau_1} + x_{\sigma_2 \cup \tau_2} \in J_{C'}$. \par

	Now we need to show that the corresponding pseudo-monomial $x_{\sigma_1} \prod_{i \in \tau_1}(1+x_i) + x_{\sigma_2}\prod_{j \in \tau_2}(1+x_j)$ is in $J_C$. Consider the $i$-th {\em bit flip map} $\delta_i$ : $\mathbb{F}_2[x_1,...,x_n] \rightarrow \mathbb{F}_2[x_1,...,x_n]$ given by 
	
\begin{equation}
       \delta_i (x_j) = 
       \begin{cases}
            x_j & \text{if $j \neq i$} \\
            1+x_j & \text{if $i=j$} \\
        \end{cases}
\end{equation}

as defined in Jeffs, Omar, and Youngs \cite{Neural_homo}, and let $\delta_{\tau_1 \cup \tau_2}$ be the composition of all $\delta_i$ such that $i \in \tau_1 \cup \tau_2$. Also as in \cite{Neural_homo}, for a neural code $C$ we define the following code:
$$\delta_i(C) := \{ u \in \mathbb{F}_2^n \mid supp(u) = supp(c) \oplus  \{i\} \text{ for some } c \in C\},$$ 
where $\oplus$ denotes the symmetric difference. Then by Theorem 2.13 in \cite{Neural_homo}, we know that $\delta_i(J_{C'}) = J_{\delta_i(C')}$ for each $i \in \tau_1 \cup \tau_2$, so $\delta_{\tau_1 \cup \tau_2} (J_{C'}) = \delta_{i_1} \circ \delta_{i_2} \circ ... \circ \delta_{i_m}(J_{C'}) = J_{\delta_{i_1} \circ \  \delta_{i_2} \circ ... \circ \  \delta_{i_m}(C')} = J_{\delta_{\tau_1 \cup \tau_2}(C')}$ for $i_1, i_2, ... i_m \in \tau_1 \cup \tau_2$. But from (\ref{eqn:U_prime}), we can see that $\delta_{\tau_1 \cup \tau_2}$ maps $C'$ into $C$, and because $\delta_{\tau_1 \cup \tau_2}(\delta_{\tau_1 \cup \tau_2}(c)) = c$ for any $c \in C'$, $\delta_{\tau_1 \cup \tau_2}$ is a bijection, which implies that $J_{\delta_{\tau_1 \cup \tau_2}(C')} = J_C$. Since each $\delta_i$ is a homomorphism \cite{Neural_homo}, $\delta_{\tau_1 \cup \tau_2}$ is also a homomorphism from $J_{C'}$ to $J_C$, so $\delta_{\tau_1 \cup \tau_2}(x_{\sigma_1 \cup \tau_1}+x_{\sigma_ \cup \tau_2}) = x_{\sigma_1} \prod_{i \in \tau_1}(1+x_i) + x_{\sigma_2}\prod_{j \in \tau_2}(1+x_j) \in J_C$. \par

	{\em Modified Type 5:} Suppose that for any $\sigma \subseteq [m]$ such that $|\sigma|$ is odd, $\cap_{k \in \sigma} U_{i_k} \subseteq \cup_{j \in [m] \setminus \sigma} U_{i_j}$. Then by the Type 2 relation, $\prod_{k \in \sigma} x_{i_k} \prod_{j \in [m] \setminus \sigma} (1+x_{i_j}) \in J_C$ for any such $\sigma$. So if $S$ is the set of all $\sigma \subseteq [m]$ such that $|\sigma|$ is odd,  then
$$f := \sum_{\sigma \in S} \Big (\prod_{k \in \sigma} x_{i_k} \prod_{j \in [m] \setminus \sigma} (1 + x_{i_j}) \Big ) \in J_C.$$ 
We claim that  $f = x_{i_1} + ... + x_{i_m}$. To see this, consider the monomial $\prod_{k \in \tau} x_{i_k}$, where $\tau \subseteq [m]$. (Note that such terms are the only possible monomials that can appear in the expansion of $f$.) We know that $\prod_{k \in \tau} x_{i_k}$ will appear exactly once in the expansion of every summand $\prod_{k \in \sigma} x_{i_k} \prod_{j \in [m] \setminus \sigma} (1 + x_{i_j})$ for which $\sigma \subseteq \tau$ and $\sigma \in S$, and never in the expansions of the other summands. Call the number of such summands $N_{\tau}$. If $|\tau| = 1$, then the only odd-sized subset of $\tau$ is $\tau$ itself, so the term $\prod_{k \in \tau} x_{i_k}$ must appear exactly once in the expansion of $f$. That is, $N_{\tau} = 1$. If $|\tau| \geq 2$, then $N_{\tau} = \sum_{1\leq i \leq |\tau|,~ i~odd} \binom{|\tau|}{i} = 2^{|\tau| -1}$, and since $|\tau| \geq 2$, $N_{\tau}$ must be even. This means that $\prod_{k \in \tau} x_{i_k}$ cancels out of $f$ when $|\tau| \geq 2$,  so the only terms that appear in the expansion of $f$ are $x_{i_1}, ..., x_{i_m}$, which implies that $f = x_{i_1} + ... + x_{i_m}$. \par

	{\em Modified Type 6:} Finally, suppose that $\cup_{k=1}^m U_{i_k} = X$ and that $U_{i_1},...,U_{i_m}$ are pairwise disjoint. Also let $M = \{i_1,...,i_m\}$. Then by the Type 3 relation we know that $\prod_{i \in M}(1+x_i) \in J_C$. But  $\prod_{i \in M}(1+x_i) = \sum_{\sigma \subseteq M} x_{\sigma}$, and since $U_{i_1},...,U_{i_m}$ are pairwise disjoint, $x_{\sigma} \in J_C$ for all $\sigma \subseteq M$ such that $|\sigma| \geq 2$ by the Type 1 relation. Therefore,
$$\sum_{\sigma \subseteq M} x_{\sigma} + \sum_{\sigma \subseteq M~,~|\sigma| \geq 2} x_{\sigma} ~=~ \sum_{\sigma \subseteq M~,~|\sigma| < 2} x_{\sigma}~=~1+x_{i_1}+...+x_{i_m} \in J_C.$$
\end{proof}

Recall that in the previous section we showed that the converses of the original Type 5 and 6 relations are false when $J_C$ is replaced by $I(C)$. However, it turns out that the converses of the modified Type 4-6 relations all hold when $J_C$ is replaced by $I(C)$.

\begin{theorem}
Let $\mathcal{U} = \{ U_i \}_{i=1}^n$ be a collection of sets in a stimulus space $X$. Let $C = C(\mathcal{U})$ denote the corresponding receptive field code, and let $I(C)$ be the ideal of $C$. Then for any subsets $\sigma_1$, $\sigma_2$, $\tau_1$, $\tau_2 \subseteq [n]$ such that $\sigma_1 \cup \sigma_2$ and $\tau_1 \cup \tau_2$ are disjoint, and $m$ indices $1 \leq i_1 < i_2 < ... < i_m \leq n$, with $m \geq 2$, we have the following equivalences:
\begin{itemize}[leftmargin=8em]
	\item[Modified Type 4:]  $x_{\sigma_1} \prod_{i \in \tau_1} (1+x_i) + x_{\sigma_2} \prod_{j \in \tau_2} (1+x_j) \in I(C) \Leftrightarrow U_{\sigma_1} \cap (\cap_{i \in \tau_1} U^c_i) = U_{\sigma_2} \cap (\cap_{j \in \tau_2} U^c_j)$
	\item[Modified Type 5:] $x_{i_1}+...+x_{i_m} \in I(C) \Leftrightarrow$ for any $\sigma \subseteq [m]$ such that $|\sigma|$ is odd, $\cap_{k \in \sigma} U_{i_k} \subseteq \cup_{j \in [m] \setminus \sigma} U_{i_j}$ \par
	\item[Modified Type 6:]  $x_{i_1} + ... + x_{i_m} + 1 \in I(C)  \Leftrightarrow \cup^m_{k=1} U_{i_k} = X$ whenever $U_{i_1},...,U_{i_m}$ are pairwise disjoint.
\end{itemize}
\end{theorem}

\begin{proof}
The backward directions of all the statements are true by Theorem 4.1 and the fact that $J_C \subseteq I(C)$. Therefore, it remains to prove the forward implications. We do this by essentially repeating the proofs in \cite{Grobner_bases}. \par
	{\em Modified Type 4:} Let $f_1 := x_{\sigma_1}\prod_{i \in \tau_1} (1+x_i)$,  let $f_2 = x_{\sigma_2}\prod_{j \in \tau_2} (1+x_j)$, and suppose that $f_1+f_2 \in I(C)$. Also suppose that $p \in  U_{\sigma_1} \cap (\cap_{i \in \tau_1} U^c_i)$. Then because $f_1+f_2 \in I(C)$, $f_1(c(p))+f_2(c(p)) = 0$, which implies that $f_1(c(p)) = f_2(c(p))$. But $p \in  U_{\sigma_k} \cap (\cap_{j \in \tau_k} U^c_j)$ if and only if $f_k(c(p))=1$, so $p \in  U_{\sigma_2} \cap (\cap_{j \in \tau_2} U^c_j)$. The same argument can be used to show that if $p \in  U_{\sigma_2} \cap (\cap_{j \in \tau_2} U^c_j)$, then $p \in U_{\sigma_1} \cap (\cap_{i \in \tau_1} U^c_i)$.  \par
	{\em Modified Type 5:}  Suppose $g:=x_{i_1}+...+x_{i_m} \in I(C)$. Let $\sigma \subseteq [m]$ such that $|\sigma|$ is odd, and let  $p \in \cap_{k \in \sigma} U_{i_k}$. (We can assume that $\cap_{k \in \sigma} U_{i_k} \neq \emptyset$ because the desired containment would be automatic otherwise.) Then because $g \in I(C)$, $g(c(p)) = 0 = c(p)_{i_1}+...+c(p)_{i_m} = \sum_{k \in \sigma} c(p)_{i_k} + \sum_{j \in [m] \setminus \sigma} c(p)_{i_j} = 1 +  \sum_{j \in [m] \setminus \sigma} c(p)_{i_j}$ in $\mathbb{F}_2$ since $|\sigma|$ is odd. Then $\sum_{j \in [m] \setminus \sigma} c(p)_{i_j}=1$, which implies that $c(p)_{i_j} =1$ for some $j \in [m] \setminus \sigma$, so $p \in U_{i_j}$ for this $j$. Therefore,  $\cap_{k \in \sigma} U_{i_k} \subseteq \cup_{j \in [m] \setminus \sigma} U_{i_j}$. \par
	{\em Modified Type 6:} Finally, suppose $h:= x_{i_1}+...+x_{i_m}+1$, and let $p \in X$. Since $h \in I(C)$, $0 = c(p)_{i_1}+...+c(p)_{i_m} + 1$. Since we are working in $\mathbb{F}_2$, this implies that for some $k \in [m]$, $c(p)_{i_k}=1$, and thus $p \in U_{i_k} \subseteq \cup^m_{k=1} U_{i_k}$. 
\end{proof}

	Directly from Theorems 4.1 and 4.2, we get the following corollary, which we view as the ``corrected'' Type 4-6 relations.

\begin{corollary}
Let $\mathcal{U} = \{ U_i \}_{i=1}^n$ be a collection of sets in a stimulus space $X$, and let $C = C(\mathcal{U})$ denote the corresponding receptive field code.  Then for any subsets $\sigma_1$, $\sigma_2$, $\tau_1$, $\tau_2 \subseteq [n]$ such that $\sigma_1 \cup \sigma_2$ and $\tau_1 \cup \tau_2$ are disjoint, and $m$ indices $1 \leq i_1 < i_2 < ... < i_m \leq n$, with $m \geq 2$, we have the following:
\begin{itemize}[leftmargin=6em]
\item[{\small Modified Type 4:}] {\small $x_{\sigma_1} \prod_{i \in \tau_1} (1+x_i) + x_{\sigma_2} \prod_{j \in \tau_2} (1+x_j) \in J_C 
	\Leftrightarrow x_{\sigma_1} \prod_{i \in \tau_1} (1+x_i) + x_{\sigma_2} \prod_{j \in \tau_2} (1+x_j) \in I(C) 
	\Leftrightarrow  U_{\sigma_1} \cap (\cap_{i \in \tau_1} U^c_i) = U_{\sigma_2} \cap (\cap_{j \in \tau_2} U^c_j).$ }
\item[{\small Modified Type 5:}]  $x_{i_1}+...+x_{i_m} \in J_C \Leftrightarrow x_{i_1}+...+x_{i_m} \in I(C) \Leftrightarrow$ for any $\sigma \subseteq [m]$ such that $|\sigma|$ is odd, $\cap_{k \in \sigma} U_{i_k} \subseteq \cup_{j \in [m] \setminus \sigma} U_{i_j}$.
\item[{\small Modified Type 6:}]  $x_{i_1} + ... + x_{i_m} + 1 \in J_C   \Leftrightarrow x_{i_1} + ... + x_{i_m} + 1 \in I(C) \Leftrightarrow \cup^m_{k=1} U_{i_k} = X$ whenever $U_{i_1},...,U_{i_m}$ are pairwise disjoint.
\end{itemize}
\end{corollary}

Essentially, Corollary 4.3 tells us that we get the same information from the modified relations regardless of whether we consider $J_C$ or $I(C)$. In fact, this is true for the Type 1-3 relations as well, as shown in \cite[Lemma 4.2]{Neural_Rings}. 

As a final note, we can use Theorem 4.1 and part of the proof of Theorem 4.2 to show that, as stated in the previous section, the {\em original} Type 4 relation is if-and-only-if when $J_C$ is replaced by $I(C)$.
\begin{corollary}
Let $\mathcal{U} = \{U_i\}^n_{i=1}$ be a collection of sets in a stimulus space $X$, and let $C = C(\mathcal{U})$ denote the corresponding receptive field code. Then for any subsets $\sigma_1$, $\sigma_2$, $\tau_1$, $\tau_2 \subseteq [n]$, we have the following:
\begin{itemize}[leftmargin=5em]
	\item[Type 4:] $x_{\sigma_1} \prod_{i \in \tau_1} (1+x_i) + x_{\sigma_2} \prod_{j \in \tau_2} (1+x_j) \in I(C) \Leftrightarrow U_{\sigma_1} \cap (\cap_{i \in \tau_1} U^c_i) = U_{\sigma_2} \cap (\cap_{j \in \tau_2} U^c_j)$.
\end{itemize}
\end{corollary}
\begin{proof}
We prove the forward direction using the same reasoning used to prove the forward direction of the modified Type 4 relation in Theorem 4.2. Thus, it remains to prove the backward direction. \par
To this end, assume that $U_{\sigma_1} \cap (\cap_{i \in \tau_1} U^c_i) = U_{\sigma_2} \cap (\cap_{j \in \tau_2} U^c_j)$, and consider the following three cases. \par
{\em Case 1 ($(\sigma_1 \cup \sigma_2) \cap (\tau_1 \cup \tau_2) = \emptyset$)}: If $\sigma_1 \cup \sigma_2$ and $\tau_1 \cup \tau_2$ are disjoint, then by Theorem 4.1, we know that $x_{\sigma_1} \prod_{i \in \tau_1} (1+x_i) + x_{\sigma_2} \prod_{j \in \tau_2} (1+x_j) \in J_C \subseteq I(C)$. \par
{\em Case 2 ($\sigma_1 \cap \tau_1 \neq \emptyset$ or $\sigma_2 \cap \tau_2 \neq \emptyset$)}: Assume without loss of generality that $\sigma_1 \cap \tau_1 \neq \emptyset$. This implies that $U_{\sigma_1} \cap (\cap_{i \in \tau_1} U^c_i) =  \emptyset$, which implies by hypothesis that $U_{\sigma_2} \cap (\cap_{i \in \tau_2} U^c_i) =  \emptyset$ as well. Now let $f_1 = x_{\sigma_1} \prod_{i \in \tau_1} (1+x_i)$ and $f_2 = x_{\sigma_2} \prod_{j \in \tau_2} (1+x_j)$. If $p \in X$, then $f_1(c(p)) = 1$ would imply that $p \in U_{\sigma_1} \cap (\cap_{i \in \tau_1} U^c_i) = \emptyset$, a contradiction. Since any codeword $c$ is associated with some $p \in X$, $f_1(c) = 0$ for all $c \in C$. Thus, $f_1 \in I(C)$. By the same reasoning, $f_2 \in I(C)$, and so $f_1+f_2 \in I(C)$ as well. \par
{\em Case 3 ($\sigma_1 \cap \tau_2 \neq \emptyset$ or $\sigma_2 \cap \tau_1 \neq \emptyset$)}: Assume without loss of generality that $\sigma_1 \cap \tau_2 \neq \emptyset$. This means that there is some $k \in \sigma_1 \cap \tau_2$. Then $U_k \supseteq  U_{\sigma_1} \cap (\cap_{i \in \tau_1} U^c_i) = U_{\sigma_2} \cap (\cap_{j \in \tau_2} U^c_j) \subseteq U_k^c$, which implies that  $U_{\sigma_1} \cap (\cap_{i \in \tau_1} U^c_i) = U_{\sigma_2} \cap (\cap_{j \in \tau_2} U^c_j) = \emptyset$. So by the same reasoning used in Case 2,  $x_{\sigma_1} \prod_{i \in \tau_1} (1+x_i) + x_{\sigma_2} \prod_{j \in \tau_2} (1+x_j) \in I(C)$.
\end{proof}

%%%%%%%%%%%%%%%%%%%%%%%%%%%%%%%%%%%%%%%%%%%%%%%%%%%%%%%%%%%%%%%%%%%%%%%%
%%%%%%%%%%%%%%%%%%%%%%%%%%%%%%%%%%%%%%%%%%%%%%%%%%%%%%%%%%%%%%%%%%%%%%%%

\section{Discussion}

In this work we proved that not only are the converses of the Type 4-6 relations in \cite{Grobner_bases} false as stated, but the converses of the Type 5 and 6 relations are also false even when the neural ideal $J_C$ is replaced by the larger ideal  $I(C)$. However, our {\em modified} versions of the Type 4-6 relations are if-and-only-if statements at the level of both $J_C$ and $I(C)$. From this we concluded that, in the case of these modified relations, $J_C$ and $I(C)$ give the same information about the stimulus space. In fact, this is true for the Type 1-3 relations as well \cite{Neural_Rings}. These observations suggest that future receptive field relationships should only involve polynomials that are in $J_C$ if and only if they are in $I(C)$. Identifying such receptive field relationships is an interesting direction for future work. 

In addition, there may be other modifications of the Type 4-6 relations that are also if-and-only-if statements. For instance, replacing $J_C$ with $I(C)$ in the original Type 4 relation also gives an if-and-only-if statement. If each of the relations could be modified in more than one way, it would be natural to ask which modifications were the ``right'' modifications; that is, which would reveal the most useful information about the corresponding receptive fields.
 
Finally, for some codes the original Type 4-6 relations are already if-and-only-if. Therefore, we can ask if all such codes have some property in common, and, conversely, we can ask if all codes with a certain property (convexity, for instance -- see \cite{Neural_Rings}) also have the property that the reverse implications of the Type 4-6 relations hold. 

\bigskip

\noindent
{\large \textbf{Acknowledgments}}  \\
This research was conducted under the mentorship of Dr. Anne Shiu in the Department of Mathematics at Texas A\&M University. The author would like to thank Dr. Shiu for all of her guidance and support, as well as Dr. Kaitlyn Phillipson for her comments on an earlier draft of this paper.

%%%%%%%%%%%%%%%%%%%%%%%%%%%%%%%%%%%%%%%%%%%%%%%%%%%%%%%%%%%%%%%%%%%%%%%%
%%%%%%%%%%%%%%%%%%%%%%%%%%%%%%%%%%%%%%%%%%%%%%%%%%%%%%%%%%%%%%%%%%%%%%%%

%References

\bibliographystyle{plain}
\bibliography{MasterBibTeX}

\begin{thebibliography}{1}

\bibitem{Neural_Rings}
Carina Curto, Vladimir Itskov, Alan Veliz-Cuba, and Nora Youngs.
\newblock The neural ring: An algebraic tool for analyzing the intrinsic
  structure of neural codes.
\newblock {\em Bulletin of Mathematical Biology}, 75:1571--1611, Sept 2013.

\bibitem{Grobner_bases}
Rebecca Garcia, Luis David~Garc\'{i}a Puente, Ryan Kruse, Jessica Liu, Dane
  Miyata, Ethan Petersen, Kaitlyn Phillipson, and Anne Shiu.
\newblock Gr\"{o}bner bases of neural ideals.
\newblock Dec 2016.
\newblock Available at {\tt arXiv:1612.05660}.

\bibitem{Neural_homo}
R.~Amzi Jeffs, Mohamed Omar, and Nora Youngs.
\newblock Neural ideal preserving homomorphisms.
\newblock Dec 2016.
\newblock Available at {\tt arXiv:1612.06150}.

\bibitem{Place_cell_discovery}
J.~O'Keefe and J.~Dostrovsky.
\newblock The hippocampus as a spatial map. {P}reliminary evidence from unit
  activity in the freely-moving rat.
\newblock {\em Brain Research}, 34:171--175, Nov 1971.

\end{thebibliography}
\end{document}